\documentclass[final,hidelinks,onefignum,onetabnum]{siamart251216}
\usepackage{amsmath,amssymb,mathtools,bm}
\usepackage{graphicx,booktabs,tabularx,array,placeins}
\usepackage{url}

\title{Efficient Matrix-Free Gauss-Newton Traveltime Tomography on Full-Metric Deformed Grids\textsuperscript{*}}
\newcommand{\affilmarkone}{\textsuperscript{\textdagger}}
\newcommand{\affilmarktwo}{\textsuperscript{\textdaggerdbl}}
\newcommand{\affilmarkthree}{\textsuperscript{\S}}
\makeatletter
\author{Zuwei Huang\affilmarkone\,\affilmarkthree
\and Pingchuan Ma\affilmarktwo
\and Peng Yu\affilmarkone
\and Takao Koyama\affilmarkthree
\and \mbox{Luolei~Zhang\affilmarkone}
\and Chongjin Zhao\affilmarkone
\and Mingxin Chu\affilmarkone}
\g@addto@macro\@thanks{\footnotetext[1]{\funding{This work was supported by the National Natural Science Foundation of China (NSFC) under Grant Nos.~42474105 and 42676055.}}}
\g@addto@macro\@thanks{\footnotetext[2]{State Key Laboratory of Marine Geology, Tongji University, Shanghai, China. E-mail addresses: Zuwei Huang (\email{hzw1498218560@tongji.edu.cn}), Peng Yu (\email{yupeng@tongji.edu.cn}), Luolei Zhang (\email{zhangluolei@hotmail.com}), Chongjin Zhao (\email{zcjadc@126.com}), and Mingxin Chu (\email{2111321@tongji.edu.cn}). Corresponding author: Peng Yu.}}
\g@addto@macro\@thanks{\footnotetext[3]{School of Mathematical Sciences, Tongji University, Shanghai, China. E-mail address: Pingchuan Ma (\email{2211180@tongji.edu.cn}).}}
\g@addto@macro\@thanks{\footnotetext[4]{Earthquake Research Institute, The University of Tokyo, Tokyo, Japan. E-mail address: Takao Koyama (\email{tkoyama@eri.u-tokyo.ac.jp}).}}
\makeatother
\headers{Matrix-Free Gauss-Newton on Deformed Grids}{Z. Huang et al.}

\begin{document}
\maketitle

\begin{abstract}
First-arrival traveltime tomography provides a computationally efficient means of imaging subsurface velocity structure, while boundary-conforming deformed grids allow rugged topography to be represented without abandoning a logically structured mesh. In matrix-free Gauss-Newton inversion, however, the linearized full-metric Eikonal transport and its transpose must be applied repeatedly at a fixed background model. Directional sweeping consequently revisits the same state-dependent dependency structure for every Krylov right-hand side. We recast the frozen tangent transport as a directed graph and distinguish loss of the background-traveltime ordering from genuine algebraic cyclicity. Strongly connected components identify the irreducible part of the transport, whereas the remaining dependencies admit exact scalar substitution after reordering. The condensation graph and local block factorizations are constructed once for each frozen source state and then reused for both primal and transpose applications. Numerical experiments on deformed grids and a three-dimensional tomography problem show that the resulting block-triangular formulation preserves the frozen sensitivity action while substantially reducing the repeated transport cost. By eliminating this dominant inner-iteration cost, the proposed method markedly accelerates Gauss-Newton computation and makes full-metric matrix-free inversion on deformed grids practical at much lower cost.
\end{abstract}

\section{Introduction}\label{introduction}

First-arrival traveltime tomography is widely used to recover subsurface velocity structure from exploration to crustal scales \cite{ref14,ref41,ref62,ref63}. Field-based formulations are especially attractive for dense three-dimensional acquisition because one source-centered Eikonal solve provides the first-arrival field throughout the computational domain and can be sampled at many receivers \cite{ref34,ref39,ref56,ref58}. This source-wise structure removes the need to solve a separate two-point ray problem for every datum and provides a natural basis for PDE-based inversion.

Rugged topography complicates that formulation. Boundary-conforming curvilinear grids retain a logically structured mesh while representing the physical surface geometrically. Transforming the isotropic Eikonal equation to such coordinates introduces a full metric \cite{ref19}, and subsequent work has developed topography-dependent forward, factored, and inversion formulations on these grids \cite{ref13,ref29,ref67,ref68,ref69}. The same metric terms that represent the geometry also couple the logical coordinate directions and thereby alter the ordering properties of the linearized transport.

Ray-free traveltime sensitivities are commonly obtained from tangent or adjoint transport equations \cite{ref15,ref24,ref25,ref31,ref38,ref43,ref49,ref53,ref61}, and first-arrival inversion can be implemented without explicit ray tracing \cite{ref23}. Such formulations are particularly efficient for first-order optimization. Matrix-free Gauss-Newton inversion has a different computational profile: an inner Krylov least-squares solve repeatedly requires both Jacobian and transpose-Jacobian actions at the same outer model \cite{ref71,ref35}. Thus avoiding explicit Jacobian storage is useful only when the repeated linearized transport solves are themselves inexpensive.

Once a nonlinear traveltime field has converged, its tangent transport is frozen during the inner Krylov solve. Repeated directional sweeping nevertheless rediscovers the same source-dependent dependency structure for every new right-hand side. On Cartesian or otherwise causally ordered stencils, marching and sweeping methods can exploit scalar update orderings directly \cite{ref3,ref9,ref40,ref54,ref55,ref66}. On a full-metric deformed grid, background traveltime need not remain a valid topological key for every discrete dependency even though continuous first-arrival causality is preserved.

The distinction between ordering failure and genuine cyclicity is central. Cyclic dependency graphs and strongly connected component (SCC) treatments have a long history in particle-transport sweeps, including parallel and distributed settings \cite{mclendon2001,mclendon2005,ref1,ref57}; block permutations have also been used in seismic tomography \cite{ref5}. These graph and sparse-matrix tools are classical. Their role here is to expose a particular structure of the frozen full-metric Eikonal sensitivity problem: traveltime-order violations can be widespread while the genuinely irreducible SCCs remain local, and the resulting block triangularization can be reused across the paired Jacobian actions required by matrix-free Gauss-Newton.

We therefore represent the frozen sensitivity matrix by its directed dependency graph. SCC condensation yields a directed acyclic graph on which singleton components are handled by scalar substitution and only nontrivial components require local block solves. The graph, ordering, and local factors are constructed once for each frozen source state and reused for both primal and transpose transport. Sections~2--4 develop this construction from the full-metric Eikonal equation to the matrix-free inversion, Section~5 presents controlled and three-dimensional numerical experiments, and Section~6 summarizes the main conclusions.

\section{Full-Metric Eikonal State and Frozen Sensitivity Transport}\label{full-metric-eikonal-state-and-frozen-sensitivity-transport}

We first derive the full-metric first-arrival equation and its tangent transport at fixed grid geometry. The resulting characteristic field is frozen during a Gauss-Newton step and provides the link between continuous first-arrival causality and the discrete ordering studied below.

\subsection{First-arrival state equation and point-source factorization}\label{first-arrival-state-equation-and-point-source-factorization}

For a point source at physical position \(\mathbf{x}_{\alpha}\), the first-arrival traveltime \(T_{\alpha}\left( \mathbf{x} \right)\) is the viscosity solution of the isotropic Eikonal equation \cite{ref7,ref47}

\begin{equation}\label{eq:eikonal}
\parallel \nabla_{\mathbf{x}}T_{\alpha}\left( \mathbf{x} \right) \parallel_{2}^{2} = s^{2}\left( \mathbf{x} \right),\quad\quad\mathbf{x} \in \Omega,\ \mathbf{x} \neq \mathbf{x}_{\alpha},
\end{equation}

with source condition \(T_{\alpha}(\mathbf{x}_{\alpha})=0\), where \(s(\mathbf{x})=1/v(\mathbf{x})\) is slowness. The source index is denoted by \(\alpha=1,\ldots,N_s\) so that \(s\) can be reserved for slowness throughout the derivation.

The point-source singularity limits the accuracy of low-order upwind formulas near the source. Following factored Eikonal methods \cite{ref9,ref54}, we write

\begin{equation}
T_{\alpha}\left( \mathbf{x} \right) = T_{0,\alpha}\left( \mathbf{x} \right)\,\tau_{\alpha}\left( \mathbf{x} \right),\quad\quad T_{0,\alpha}\left( \mathbf{x} \right) = s_{\alpha} \parallel \mathbf{x} - \mathbf{x}_{\alpha} \parallel_{2},
\end{equation}

where \(s_{\alpha}\) is the source slowness and \(\tau_{\alpha}\) is smoother than the full traveltime. The nonlinear state calculation uses the corresponding full-metric factored Eikonal equation on a boundary-conforming logically structured grid, consistent with established topography-dependent formulations \cite{ref19,ref67,ref68}. In what follows, the state defined by \cref{eq:eikonal} is treated as converged; the development concerns the linear transport that must then be applied repeatedly inside the inversion.

\subsection{Boundary-conforming coordinates and the full metric}\label{boundary-conforming-coordinates-and-the-full-metric}

Let the logical coordinates be \(\mathbf{\xi}=(\xi,\eta,\zeta)^T\), and let the smooth, nonsingular mapping \(\mathbf{x}=\mathbf{x}(\mathbf{\xi})\) map the logical grid to the physical domain. Its Jacobian is \(J=\partial\mathbf{x}/\partial\mathbf{\xi}\), with \(J_{ij}=\partial x_i/\partial\xi_j\). For any scalar field \(T\), the chain rule gives \(\nabla_{\mathbf{\xi}}T=J^T\nabla_{\mathbf{x}}T\), and hence \(\nabla_{\mathbf{x}}T=J^{-T}\nabla_{\mathbf{\xi}}T\).

Substitution of these relations into \cref{eq:eikonal} gives

\begin{equation}\label{eq:metric-eikonal}
\left( \nabla_{\mathbf{\xi}}T \right)^{T}G\nabla_{\mathbf{\xi}}T = s^{2},\quad\quad G = J^{- 1}J^{- T}.
\end{equation}

For a nonsingular mapping, \(G\) is symmetric positive definite because, for every nonzero vector \(\mathbf{z}\),

\begin{equation}
\mathbf{z}^{T}G\mathbf{z} = \mathbf{z}^{T}J^{- 1}J^{- T}\mathbf{z} = \parallel J^{- T}\mathbf{z} \parallel_{2}^{2} > 0.
\end{equation}

Writing the metric explicitly,

\begin{equation}
G = \begin{bmatrix}
G_{\xi\xi} & G_{\xi\eta} & G_{\xi\zeta} \\
G_{\xi\eta} & G_{\eta\eta} & G_{\eta\zeta} \\
G_{\xi\zeta} & G_{\eta\zeta} & G_{\zeta\zeta}
\end{bmatrix},
\end{equation}

shows how nonorthogonality couples logical directions. This coupling is geometric, not a physical anisotropy: an isotropic physical Eikonal equation acquires an anisotropic-like representation after a nonorthogonal coordinate transformation \cite{ref19}.

A geometric interpretation is useful later. Define the covariant coordinate basis \(\mathbf{g}_i=\partial\mathbf{x}/\partial\xi_i\) and the reciprocal basis \(\mathbf{g}^i=\nabla_{\mathbf{x}}\xi_i\), with \(\mathbf{g}^i\cdot\mathbf{g}_j=\delta_j^i\). Then \(T_i=\partial T/\partial\xi_i=\mathbf{g}_i\cdot\nabla_{\mathbf{x}}T\) is a covariant derivative, while \(G_{ij}=\mathbf{g}^i\cdot\mathbf{g}^j\).

On an orthogonal grid the covariant and reciprocal directions are parallel. On a skew grid they are not, and this distinction is what later permits a logical characteristic component to have a sign different from the corresponding logical traveltime derivative.

\subsection{Continuous linearization at fixed geometry}\label{continuous-linearization-at-fixed-geometry}

The inversion model variable is subsurface velocity; the grid geometry is held fixed during a Gauss-Newton step, so \(\delta J=0\) and \(\delta G=0\). This assumption is essential. If geometry itself were an inversion variable, an additional term involving \(\delta G\) would appear in the tangent equation.

Define the full-metric Eikonal residual

\begin{equation}
\mathcal{E}(T,s) = \frac{1}{2}\left\lbrack \left( \nabla_{\mathbf{\xi}}T \right)^{T}G\nabla_{\mathbf{\xi}}T - s^{2} \right\rbrack.
\end{equation}

Taking the first variation at fixed \(G\) gives

\begin{equation}
\delta\mathcal{E} = \left( G\nabla_{\mathbf{\xi}}T \right)^{T}\nabla_{\mathbf{\xi}}\delta T - s\,\delta s.
\end{equation}

Define the metric-weighted characteristic vector \(\mathbf{a}=G\nabla_{\mathbf{\xi}}T\). Using this definition in the first variation gives the frozen tangent equation

\begin{equation}\label{eq:tangent}
\mathbf{a} \cdot \nabla_{\mathbf{\xi}}\delta T = s\,\delta s.
\end{equation}

Equation~\eqref{eq:tangent} is the continuous Fr\'echet derivative of the Eikonal state equation about the converged first arrival: a slowness perturbation supplies the source term and the corresponding traveltime perturbation is transported along the frozen characteristic field. We follow the established topography-dependent framework in which the nonlinear state is computed with the three-dimensional FTDEE solver \cite{ref68}, while tangent and adjoint actions are obtained from this linearized full-metric transport \cite{ref13,ref69,ref72}. The contribution below is the reusable algebraic solution of that transport, not a new state or sensitivity equation.

Because \(s=v^{-1}\), \(\delta s=-v^{-2}\delta v=-s^2\delta v\); for a relative velocity perturbation \(m=\delta v/v\), this becomes \(\delta s=-sm\). These relations will be used to define the model-to-transport map in Section 4.

\subsection{Physical causality and scalar ordering}\label{physical-characteristic-causality-versus-componentwise-monotonicity}

Before discretizing \cref{eq:tangent}, we distinguish physical first-arrival causality from the scalar ordering used by a discrete sweep. The full-metric equation preserves the former. Parameterize a continuous characteristic in logical coordinates by \(d\mathbf{\xi}/d\tau=\mathbf{a}\). Along this characteristic,

\begin{equation}\label{eq:causality}
\frac{dT}{d\tau} = \nabla_{\mathbf{\xi}}T \cdot \frac{d\mathbf{\xi}}{d\tau} = \left( \nabla_{\mathbf{\xi}}T \right)^{T}G\nabla_{\mathbf{\xi}}T = s^{2} > 0.
\end{equation}

Thus \cref{eq:causality} shows that a nonsingular coordinate transformation does not create a closed physical first-arrival ray or reverse propagation in time. This identity is the reference point for interpreting the discrete graph: any directed cycle found later is an algebraic cycle among coordinate-wise dependencies, not a physical causal loop.

The stronger property needed for a scalar traveltime ordering is componentwise. On an orthogonal grid, \(G=\operatorname{diag}(g^{11},g^{22},g^{33})\), where \(g^{ii}>0\), and therefore \(a^i=g^{ii}T_i\) and \(T_i a^i=g^{ii}T_i^2\geq0\). Hence the direction selected by each logical component is consistent with the sign of the corresponding logical traveltime derivative.

On a full metric, by contrast,

\begin{equation}
a^{\xi} = G_{\xi\xi}T_{\xi} + G_{\xi\eta}T_{\eta} + G_{\xi\zeta}T_{\zeta},
\end{equation}

with analogous formulas for \(a^{\eta}\) and \(a^{\zeta}\). The Eikonal identity guarantees only the sum

\begin{equation}
T_{\xi}a^{\xi} + T_{\eta}a^{\eta} + T_{\zeta}a^{\zeta} = s^{2} > 0,
\end{equation}

not the sign of each term separately. A coordinate-wise reversal in the \(\xi\) direction is therefore possible when \(T_{\xi}a^{\xi}<0\), or equivalently

\begin{equation}\label{eq:reversal}
T_{\xi}\left( G_{\xi\eta}T_{\eta} + G_{\xi\zeta}T_{\zeta} \right) < - G_{\xi\xi}T_{\xi}^{2}.
\end{equation}

The reversal condition in \cref{eq:reversal} isolates the mechanism. The off-diagonal metric alone is not sufficient, and a complicated wavefront alone is not sufficient. Reversal depends on their interaction: grid geometry determines the cross-metric coupling, while the background wavefront determines the local components of \(\nabla_{\mathbf{\xi}}T\). It is particularly easy for a cross term to change the sign of \(a^{i}\) where the corresponding diagonal contribution \(g^{ii}T_{i}\) is small. This is a mathematical possibility of the full-metric representation; the frequency with which it occurs, and whether several reversals form a closed graph cycle, are problem-dependent quantities to be measured rather than assumed.

\section{Frozen Discrete Transport and the Tomography Dependency Graph}\label{frozen-discrete-transport-and-the-tomography-dependency-graph}

Equation~\eqref{eq:tangent} describes transport along the continuous characteristic, but the cost of a repeated sensitivity solve is determined by the dependency structure of its discrete counterpart. This section constructs that discrete operator and then reads its nonzero pattern as a directed graph. The graph viewpoint makes it possible to distinguish a poor scalar ordering from genuine algebraic feedback, a distinction that is obscured when every right-hand side is handled by repeated geometric sweeps.

\subsection{Face-based upwinding}\label{face-based-upwind-transport}

The tangent equation is discretized after convergence of the nonlinear FTDEE state. We use the face-based upwind construction associated with established 3-D full-metric adjoint transport \cite{ref13,ref72}, written here in primal form because its dependency graph is the object of interest.

With \(\mathbf a=G\nabla_{\boldsymbol\xi}T\), a face normal to each logical direction uses the corresponding interpolated signed coefficient \(-a_f^{\xi}\), \(-a_f^{\eta}\), or \(-a_f^{\zeta}\). Its sign selects the upstream side and its magnitude supplies the coupling strength. Metric coefficients and traveltime derivatives use the face averages of the established 3-D discretization \cite{ref72}.

For an interior node \(i\), let \(U(i)\) contain the selected upstream neighbors and let \(c_{ij}\geq0\) denote their coupling magnitudes. Define \(d_i=\sum_{j\in U(i)}c_{ij}\) and \(w_{ij}=c_{ij}/d_i\). After normalization, an interior row has the form

\begin{equation}\label{eq:discrete-row}
\delta T_{i} - \sum_{j \in U(i)}^{}w_{ij}\delta T_{j} = q_{i},
\end{equation}

Collecting the rows gives the source-dependent linear system

\begin{equation}\label{eq:transport-system}
A\,\delta\mathbf{T} = \mathbf{q},\quad\quad A = I - W.
\end{equation}

On the structured hexahedral stencil, an interior row has at most six geometric neighbor dependencies. Source-seed rows are treated separately from the generic interior normalization. The directed graph introduced below is constructed from these actual discrete coefficients, not from geometric proximity and not from a post-processed ray field.

\subsection{Scale-consistent transport right-hand side}\label{scale-consistent-transport-right-hand-side}

The nonlinear FTDEE state and the full-metric tangent transport follow the continuous-then-discretize construction used in topography-dependent adjoint tomography \cite{ref13,ref68,ref69,ref72}. Because the factored nonlinear state and the linear transport use different discrete forms, the implementation calibrates the model-to-transport right-hand side so that a uniform relative velocity scaling is reproduced exactly by the frozen operator.

Let \(\boldsymbol\ell=A\mathbf T\), and for nonsource rows define \(c_i^{(s)}=\ell_i/s_i\). The interior transport source associated with a nodal slowness perturbation is \(q_i=c_i^{(s)}\delta s_i\), with source rows supplied by the differentiated source treatment. For the uniform relative perturbation \(\delta v=v\), we have \(\delta s=-s\) and hence \(\mathbf q=-A\mathbf T\). If \(A\) is nonsingular, \(A\delta\mathbf T=-A\mathbf T\) implies \(\delta\mathbf T=-\mathbf T\), which is the exact first-order scaling of traveltime under a uniform relative velocity change. The calibration changes only the right-hand-side map. It does not modify \(A\), its sparsity pattern, its dependency graph, or the relation between primal and transpose transport.

\subsection{From a sensitivity equation to a directed graph}\label{from-a-sensitivity-equation-to-a-directed-graph}

The matrix in \cref{eq:transport-system} defines a directed dependency graph \(\mathcal G(A)=(V,E)\), with one vertex for each transport unknown and an edge \(j\rightarrow i\) if and only if \(w_{ij}>0\). The edge has a direct algebraic meaning: the equation for unknown \(i\) requires the value at unknown \(j\). The nonzero pattern of the frozen transport therefore records the ordering constraints of the sensitivity solve.

\begin{proposition}[Graph acyclicity and scalar triangularization]\label{prop:dag-triangular}
Assume the diagonal entries of $A$ are nonzero. Then
\begin{equation}\label{eq:dag-triangular}
\mathcal{G}(A)\text{ is a DAG}\quad \Longleftrightarrow \quad
\Pi A\Pi^{T}\text{ is triangular for some permutation }\Pi.
\end{equation}
\begin{proof}
If $\mathcal{G}(A)$ is acyclic, it admits a topological ordering. Permuting the unknowns by that ordering places every directed dependency on the same side of the diagonal, so the permuted matrix is triangular. Conversely, the directed graph of a triangular matrix cannot contain a directed cycle, because a cycle would require at least one dependency to cross the diagonal in the forbidden direction.
\end{proof}
\end{proposition}

Proposition~\ref{prop:dag-triangular} gives a precise interpretation of scalar causality for the frozen sensitivity system: one-pass substitution is possible if and only if the dependency graph is acyclic, regardless of whether the valid order is the background traveltime order.

\subsection{Traveltime is a special topological key, not the definition of causality}\label{traveltime-is-a-special-topological-key-not-the-definition-of-causality}

On an orthogonal monotone stencil, Section 2.4 gives sign consistency between \(T_i\) and \(a^i\). Away from equal-time degeneracies, the sufficient condition \(j\rightarrow i\Rightarrow T_j<T_i\) holds. Sorting nodes by increasing background traveltime then supplies a topological order. This is the discrete algebraic counterpart of the causal ordering exploited by marching methods and by triangular sensitivity calculations when a causal fast marching order is available \cite{ref47,ref54,ref70}.

On a full metric, the face-upwind decision follows the sign of \(a^{i}\), whereas the proposed traveltime key is determined by \(T_{i}\). If \(T_{i}a^{i} < 0\), a selected face dependency can oppose the local scalar traveltime order. We call such an edge \emph{traveltime-order violating}. The term \emph{noncausal edge} may be used only in this restricted algebraic sense; physical first-arrival causality remains intact by Section 2.4.

Crucially, the existence of an edge \(j \rightarrow i\) with \(T_j>T_i\) means only that traveltime is not a valid topological key. It does not imply that \(\mathcal G(A)\) contains a cycle. By \cref{prop:dag-triangular}, an acyclic graph with such edges still admits another scalar ordering and remains exactly solvable by substitution. The hierarchy relevant to the solver is therefore \(\text{traveltime-ordered DAG}\subset\text{general DAG}\subset\text{general directed graph}\), and only the last class can contain irreducible cyclic subsets.

\subsection{Why a local cycle can obstruct a large scalar elimination region}\label{why-a-local-cycle-can-obstruct-a-large-scalar-elimination-region}

The distinction between a small cyclic core and a large unresolved scalar region is a graph property, not a convergence-tolerance effect. Consider Kahn's topological elimination: vertices with zero unresolved indegree are removed recursively. Let \(R\) denote the set remaining when no further vertex can be removed.

\begin{proposition}[Downstream closure of the Kahn residual]\label{prop:kahn}
For a finite directed graph, every vertex in $R$ lies on a directed cycle or is reachable from a directed cycle contained in $R$. Conversely, every vertex reachable from a directed cycle remains unresolved by Kahn elimination.
\begin{proof}
Every vertex in the residual subgraph has at least one predecessor in $R$. Starting from any $v\in R$ and repeatedly following a predecessor yields, by finiteness, a repeated vertex and hence a directed cycle. Reversing this predecessor chain gives a directed path from that cycle to $v$. Conversely, vertices on a directed cycle can never reach zero unresolved indegree. Any vertex reachable from such a cycle retains at least one unresolved predecessor along the path and therefore also cannot be removed.
\end{proof}
\end{proposition}

Proposition~\ref{prop:kahn} explains a central computational effect. A nontrivial SCC may occupy only a small local region, yet every downstream vertex that depends on it is unavailable to scalar topological substitution until the feedback is resolved. Thus the fraction of vertices unresolved by scalar elimination can be much larger than the fraction actually belonging to cyclic SCCs. This amplification is structural and is independent of the stopping criterion later used by an iterative transport solver.

\subsection{Repeated directional sweeping as an iterative matrix splitting}\label{repeated-directional-sweeping-as-an-iterative-matrix-splitting}

The graph analysis above also clarifies what directional sweeping is doing when its geometric order is not topological. Directional sweeping remains a valid iterative strategy, but it then acts as a stationary iteration for the frozen linear system. For a fixed scalar ordering \(\sigma\), write \(A=M_{\sigma}-N_{\sigma}\), where \(M_{\sigma}\) contains the diagonal and dependencies available in the current Gauss-Seidel pass. One sweep is

\begin{equation}
\mathbf{x}^{(k + 1)} = M_{\sigma}^{- 1}N_{\sigma}\mathbf{x}^{(k)} + M_{\sigma}^{- 1}\mathbf{b}.
\end{equation}

The iteration matrix is \(G_{\sigma}=M_{\sigma}^{-1}N_{\sigma}\). If \(\sigma\) is an exact topological order of an acyclic graph, then all dependencies are available in one pass, \(N_{\sigma}=0\), and therefore \(G_{\sigma}=0\). The sweep is then identical to direct scalar substitution. If the graph is acyclic but the imposed geometric sweep order is not the exact graph order, \(G_{\sigma}\neq0\) and additional passes are required even though no SCC exists. If a nontrivial SCC is present, no scalar ordering can place all internal mutual dependencies in a one-pass triangular part; some internal fixed-point iteration remains unless the component is treated as a block.

Hence extra sweeping work comes from ordering mismatch on the acyclic graph and true feedback inside irreducible SCCs. This distinction is important for both interpretation and benchmarking. Related transport-sweep literature similarly separates scheduling/order effects from cyclic mesh dependencies \cite{mclendon2001,mclendon2005,ref1,ref57}. For matrix-free Gauss-Newton, however, both effects are revisited for every right-hand side even though the frozen matrix \(A\) is unchanged. This motivates exposing the reusable ordering once and treating only the irreducible part as a block.

\section{Reusable SCC/BTF Transport for Matrix-Free Gauss-Newton Inversion}\label{reusable-block-triangular-transport-for-matrix-free-gauss-newton}

The preceding analysis suggests that the frozen transport should not be iterated globally when most of its graph is already scalar-reorderable. We instead expose that structure once, condense only the irreducible dependencies into blocks, and reuse the resulting block triangular form for every right-hand side generated by the inner Krylov method. The construction below is applied independently to each source after its nonlinear traveltime state has converged.

\subsection{Strongly connected components and exact block triangularization}\label{strongly-connected-components-and-exact-block-triangularization}

Strongly connected components are the maximal subsets of vertices for which every pair is mutually reachable \cite{ref50}. The SCC partition is unique up to component ordering. Collapsing every SCC into one supernode produces the \emph{condensation graph}, which is always a DAG. Classical sparse-matrix work uses this graph structure to obtain block triangular forms \cite{ref8}.

Let $\Pi$ order the SCC condensation graph topologically. The permuted operator then has the block lower-triangular form

\begin{equation}\label{eq:block-form}
\widehat{A} = \Pi A\Pi^{T} = \begin{bmatrix}
A_{11} & 0 & \cdots & 0 \\
A_{21} & A_{22} & \cdots & 0 \\
 \vdots & \vdots & \ddots & \vdots \\
A_{m1} & A_{m2} & \cdots & A_{mm}
\end{bmatrix}.
\end{equation}

Each diagonal block \(A_{kk}\) corresponds to one SCC. A singleton SCC yields the scalar block \(\lbrack 1\rbrack\) under the normalized transport. A nontrivial SCC is precisely a set that cannot be fully triangularized by any scalar permutation. In this sense the SCC decomposition is not an arbitrary block partition: it gives the exact maximal scalar-reorderable structure of the frozen transport.

For the \(k\)th block row, \(A_{kk}\mathbf{x}_k=\mathbf{b}_k-\sum_{\ell<k}A_{k\ell}\mathbf{x}_{\ell}\), and all predecessor blocks are known in topological order.

\begin{proposition}[Exact block traversal]\label{prop:block-traversal}
If every diagonal SCC block in \cref{eq:block-form} is nonsingular, a topological block traversal solves the full frozen system exactly up to floating-point factorization and substitution error.
\end{proposition}
The result follows directly from block forward substitution: once the diagonal SCC blocks are solved, no global fixed-point iteration is required.

\subsection{A sufficient nonsingularity condition for local cyclic blocks}\label{a-sufficient-nonsingularity-condition-for-local-cyclic-blocks}

For a nontrivial SCC \(C\), the normalized local block has the form \(A_C=I-W_{CC}\), where \(W_{CC}\geq0\). Because the full row weights are normalized, \(\sum_{j\in C}w_{ij}\leq1\); the inequality is strict for a row that also receives positive dependency weight from a predecessor outside \(C\).

With an edge for every positive dependency weight, the SCC property makes \(W_{CC}\) irreducible (edge reversal does not change strongly connected components). Moreover, restricting a row to \(C\) removes every positive weight carried by an external predecessor, which gives the strict internal row-sum inequality required below.

\begin{theorem}[Sufficient local invertibility]\label{thm:local-invertibility}
Suppose $W_{CC}$ is nonnegative and irreducible, every row sum is at most one, and at least one row sum is strictly smaller than one. Then
\begin{equation}\label{eq:rho-local}
\rho(W_{CC})<1,
\end{equation}
and hence $A_C=I-W_{CC}$ is a nonsingular M-matrix with
\begin{equation}\label{eq:neumann-local}
A_C^{-1}=\sum_{n=0}^{\infty}W_{CC}^{n}\ge 0.
\end{equation}
\begin{proof}
The row-sum bound gives $\rho(W_{CC})\le \|W_{CC}\|_\infty\le1$. Assume for contradiction that $\rho(W_{CC})=1$. By Perron-Frobenius theory for irreducible nonnegative matrices \cite{ref4}, there exists $\mathbf z>0$ such that $W_{CC}\mathbf z=\mathbf z$. Let $M=\max_i z_i$ and \(S=\{i:z_i=M\}\). For any \(i\in S\),
\begin{equation}
M=\sum_j w_{ij}z_j\le M\sum_j w_{ij}\le M.
\end{equation}
Both inequalities must be equalities. Hence the \(i\)th row sum is one and every \(j\) with \(w_{ij}>0\) also belongs to \(S\). Thus \(S\) is a nonempty successor-closed set in the directed graph of \(W_{CC}\), where \(i\to j\) when \(w_{ij}>0\). Irreducibility implies that every vertex is reachable from any \(i\in S\), so \(S=C\). Therefore \(\mathbf z=M\mathbf 1\), and \(W_{CC}\mathbf z=\mathbf z\) forces every row sum to equal one, contradicting the strict inequality in at least one row. Thus \cref{eq:rho-local} holds, and the Neumann series in \cref{eq:neumann-local} converges and is nonnegative.
\end{proof}
\end{theorem}

The theorem is sufficient, not necessary. The method does not assume that every block encountered in every problem satisfies this condition; local conditioning and factorization success are numerical quantities that must be checked in experiments.

\begin{corollary}\label{cor:global-invertibility}
If every diagonal SCC block is nonsingular, then the global frozen operator $A$ is nonsingular.
\end{corollary}
Indeed, \(\det(\widehat A)=\prod_{k=1}^{m}\det(A_{kk})\neq0\), and $\widehat A$ is a permutation of $A$.

\subsection{Primal and transpose block solves with reusable factors}\label{primal-and-transpose-block-solves-with-reusable-factors}

Singleton components require only scalar substitution. For each nontrivial component \(C\), factorize once with pivoting as \(P_CA_C=L_CU_C\). Every subsequent primal right-hand side reuses this factorization. The transpose of the globally permuted system is block upper triangular, so the SCC condensation graph is traversed in reverse. From the stored factorization, \(A_C^T=U_C^TL_C^TP_C\); thus the local problem \(A_C^T\mathbf y=\mathbf r\) is solved by triangular substitutions with \(U_C^T\) and \(L_C^T\) followed by the stored permutation. No factorization of \(A_C^T\) is needed. This reuse is especially important for Golub-Kahan Krylov least-squares methods, whose inner iteration alternates applications of an operator and its transpose \cite{ref71}.

For source \(\alpha\), define the reusable state-dependent setup

\begin{equation}
\mathcal{S}_{\alpha} = \left\{ T_{\alpha},A_{\alpha},\mathcal{G}_{\alpha},\Pi_{\alpha},\{ P_{C},L_{C},U_{C}\}_{C \in \mathcal{C}_{\alpha}} \right\}.
\end{equation}

Once \(\mathcal S_{\alpha}\) has been built for the current outer model, different Krylov right-hand sides require only \(\mathbf q^{(m)}\mapsto A_{\alpha}^{-1}\mathbf q^{(m)}\) and \(\mathbf r^{(m)}\mapsto A_{\alpha}^{-T}\mathbf r^{(m)}\). This setup-apply separation is the main reason the graph representation is valuable for matrix-free Gauss-Newton rather than merely for a single sensitivity calculation.

\subsection{Paired matrix-free Jacobian and transpose actions}\label{paired-matrix-free-jacobian-and-transpose-actions}

Let \(M_{\alpha}\) map a velocity perturbation to the nodal slowness perturbation used by source \(\alpha\), \(C_{\alpha}\) map that slowness perturbation to the calibrated transport right-hand side, and \(P_{\alpha}\) sample the transported traveltime perturbation at the source's receivers. The implementation also contains a source-local direct data term \(D_{\alpha}\) for source configurations in which the differentiated source treatment contributes directly to receiver interpolation. With these maps, the source Jacobian can be written as

\begin{equation}\label{eq:source-jacobian}
J_{\alpha} = \left( P_{\alpha}A_{\alpha}^{- 1}C_{\alpha} + D_{\alpha} \right)M_{\alpha}.
\end{equation}

For common off-grid source configurations, \(D_{\alpha} = 0\). The direct term is retained here so that the operator identity matches the general source treatment rather than only the most common case.

Given a model-space vector \(\mathbf p\), the matrix-free product first forms \(\delta\mathbf s_{\alpha}=M_{\alpha}\mathbf p\) and \(\mathbf q_{\alpha}=C_{\alpha}\delta\mathbf s_{\alpha}\), solves \(A_{\alpha}\delta\mathbf T_{\alpha}=\mathbf q_{\alpha}\), and returns \(J_{\alpha}\mathbf p=P_{\alpha}\delta\mathbf T_{\alpha}+D_{\alpha}\delta\mathbf s_{\alpha}\). For a data-space vector \(\mathbf y_{\alpha}\), define the adjoint transport variable by \(A_{\alpha}^T\boldsymbol\lambda_{\alpha}=P_{\alpha}^T\mathbf y_{\alpha}\). Taking the algebraic transpose of the same discrete maps gives

\begin{equation}\label{eq:source-transpose}
J_{\alpha}^{T}\mathbf{y}_{\alpha} = M_{\alpha}^{T}\left( C_{\alpha}^{T}\mathbf{\lambda}_{\alpha} + D_{\alpha}^{T}\mathbf{y}_{\alpha} \right).
\end{equation}

The pairing follows directly from the discrete maps. For arbitrary \(\mathbf p\) and \(\mathbf y_{\alpha}\),
\begin{align*}
\langle J_{\alpha}\mathbf p,\mathbf y_{\alpha}\rangle
&=\left\langle P_{\alpha}A_{\alpha}^{-1}C_{\alpha}M_{\alpha}\mathbf p,\mathbf y_{\alpha}\right\rangle
 +\left\langle D_{\alpha}M_{\alpha}\mathbf p,\mathbf y_{\alpha}\right\rangle \\
&=\left\langle C_{\alpha}M_{\alpha}\mathbf p,A_{\alpha}^{-T}P_{\alpha}^T\mathbf y_{\alpha}\right\rangle
 +\left\langle M_{\alpha}\mathbf p,D_{\alpha}^T\mathbf y_{\alpha}\right\rangle \\
&=\left\langle\mathbf p,M_{\alpha}^T\left(C_{\alpha}^TA_{\alpha}^{-T}P_{\alpha}^T\mathbf y_{\alpha}+D_{\alpha}^T\mathbf y_{\alpha}\right)\right\rangle
=\langle\mathbf p,J_{\alpha}^T\mathbf y_{\alpha}\rangle.
\end{align*}

Thus \cref{eq:source-jacobian,eq:source-transpose} define the paired discrete Jacobian and transpose-Jacobian actions used below. The former discretizes the continuous Fr\'echet derivative in \cref{eq:tangent} following the topography-dependent tangent/adjoint framework \cite{ref13,ref24,ref31,ref49,ref69,ref72}; the latter is the exact algebraic transpose of the same discrete maps, subject only to floating-point roundoff. Fully discrete adjoint formulations provide an alternative discretization route \cite{ref70}.

For \(N_s\) sources, the global Jacobian and transpose action are
\begin{equation}
J=\begin{bmatrix}J_1\\ \vdots\\ J_{N_s}\end{bmatrix},
\qquad
J^T\mathbf y=\sum_{\alpha=1}^{N_s}J_{\alpha}^T\mathbf y_{\alpha}.
\end{equation}
Thus \(J\mathbf p\) concatenates the source products, and the source-wise graph setup is naturally reusable inside a distributed matrix-free implementation.

\subsection{Matrix-free Gauss-Newton-Krylov inversion}\label{regularized-gauss-newton-and-the-augmented-krylov-least-squares-operator}

Let $F(\mathbf v)$ denote the nonlinear multi-source first-arrival map and $\mathbf d^{obs}$ the observed traveltimes. At outer iteration $k$,
\begin{equation}\label{eq:residual}
\mathbf r_k=F(\mathbf v_k)-\mathbf d^{obs}.
\end{equation}
Let $W_d$ be the diagonal data-weighting matrix, $R$ the model-increment roughness operator, and
\begin{equation}\label{eq:relative-damping}
S_k=\gamma\,\operatorname{diag}(\mathbf v_k^{-1})
\end{equation}
be a relative increment damping operator. With trade-off parameter $\lambda>0$, the stabilized Gauss-Newton increment is
\begin{equation}\label{eq:gn-subproblem}
\delta\mathbf v_k=\arg\min_{\delta\mathbf v}\left\{
\frac12\|W_d(J_k\delta\mathbf v+\mathbf r_k)\|_2^2+
\frac{\lambda^2}{2}\|R\delta\mathbf v\|_2^2+
\frac{\lambda^2}{2}\|S_k\delta\mathbf v\|_2^2\right\}.
\end{equation}
Neither $J_k$ nor a normal matrix is formed. LSMR \cite{ref71} is applied to the augmented least-squares system
\begin{equation}\label{eq:augmented-ls}
\min_{\delta\mathbf v}\left\|
\begin{bmatrix}W_dJ_k\\ \lambda R\\ \lambda S_k\end{bmatrix}\delta\mathbf v-
\begin{bmatrix}-W_d\mathbf r_k\\0\\0\end{bmatrix}\right\|_2.
\end{equation}
The corresponding matrix-free operator and transpose are
\begin{equation}\label{eq:augmented-actions}
\mathcal K_k\mathbf p=
\begin{bmatrix}W_dJ_k\mathbf p\\ \lambda R\mathbf p\\ \lambda S_k\mathbf p\end{bmatrix},\qquad
\mathcal K_k^T
\begin{bmatrix}\mathbf y\\\mathbf z_R\\\mathbf z_S\end{bmatrix}
=J_k^TW_d^T\mathbf y+\lambda R^T\mathbf z_R+\lambda S_k^T\mathbf z_S.
\end{equation}
Thus each Krylov iteration repeatedly invokes $J_k$ and $J_k^T$ at the same frozen state, which is precisely the regime in which the graph/SCC setup can be reused.

After the linearized step is computed,
\begin{equation}\label{eq:model-update}
\mathbf v_{k+1}=\mathbf v_k+t_k\delta\mathbf v_k,
\end{equation}
where $t_k$ is selected by an Armijo line search evaluated with the nonlinear weighted data misfit. The roughness and relative-damping terms stabilize the increment in \cref{eq:gn-subproblem}; the line search itself is evaluated through the nonlinear forward problem \cite{ref35}.

\subsection{Setup and repeated-application complexity}\label{setup-and-repeated-application-complexity}

Let a source-dependent graph contain \(N\) transport unknowns and \(E\) directed dependencies, and let the sizes of its nontrivial SCCs be \(k_{c}\). Graph construction, SCC decomposition, and condensation ordering are linear in graph size \cite{ref50}. If the local cyclic blocks are factorized densely, the one-time setup cost has the form

\begin{equation}\label{eq:setup-cost}
C_{setup} = O(N + E) + O\left( \sum_{c}^{}k_{c}^{3} \right).
\end{equation}

A subsequent primal or transpose application requires graph traversal plus local triangular substitutions,

\begin{equation}\label{eq:apply-cost}
C_{apply} = O(N + E) + O\left( \sum_{c}^{}k_{c}^{2} \right).
\end{equation}

If one outer Gauss-Newton state requires \(m\) tangent or transpose right-hand sides, the block method and repeated sweeping have the respective amortized forms
\begin{equation}
C_{block}(m)=C_{setup}+mC_{apply},
\qquad
C_{sweep}(m)\approx m\,n_{pass}\,C_{pass},
\end{equation}

where \(n_{pass}\) depends on geometric ordering mismatch, true feedback inside SCCs, the right-hand side, and the stopping criterion. We do not assume a universal value or convergence rate for \(n_{pass}\); it is an empirical solver quantity. The graph formulation instead removes ordering mismatch exactly and confines non-scalar work to the irreducible blocks.

The favorable regime is therefore one in which SCC condensation leaves a predominantly scalar DAG and the nontrivial blocks remain local enough for their setup to be amortized over repeated $J/J^T$ applications. Section~5 examines this regime directly.

\subsection{Overall algorithm}\label{overall-algorithm}

Let \(\mathcal B_{\alpha,k}\) and \(\mathcal B_{\alpha,k}^{T}\) denote the stored primal and reverse SCC/BTF traversals that apply \(A_{\alpha,k}^{-1}\) and \(A_{\alpha,k}^{-T}\), respectively. The complete workflow has two nested time scales. At each outer iteration, the nonlinear state, source-wise maps, dependency graph, SCC ordering, and local factors are constructed once. The inner LSMR iterations then reuse those frozen objects for every paired Jacobian and transpose-Jacobian application, as summarized in Algorithm~\ref{alg:overall-gn}.

\begin{algorithm}[H]
\caption{End-to-end matrix-free Gauss-Newton tomography with reusable SCC/BTF transport}
\label{alg:overall-gn}
\footnotesize
\begin{tabularx}{\linewidth}{@{}r@{\hspace{0.65em}}X@{}}
1 & \textbf{Input:} \(\mathbf v_0,\mathbf d^{obs},W_d,R,\lambda,\gamma,k_{\max}\); deformed grid; acquisition geometry; outer and LSMR stopping tolerances.\\
2 & \textbf{Output:} the final accepted velocity model \(\mathbf v\).\\
3 & \textbf{for} \(k=0,1,\ldots,k_{\max}\) \textbf{do}\\
4 & \quad \textit{Phase I: nonlinear state and reusable source-wise setup.}\\
5 & \quad \textbf{for} \(\alpha=1,\ldots,N_s\) \textbf{do}\\
6 & \qquad Solve \(T_{\alpha,k}\leftarrow\operatorname{FTDEE}(\mathbf v_k)\).\\
7 & \qquad Assemble \(A_{\alpha,k}=I-W_{\alpha,k}\), and freeze the maps \((C,M,P,D)_{\alpha,k}\).\\
8 & \qquad Build \(\mathcal G_{\alpha,k}=\{j\to i:w_{ij}>0\}\); compute its SCCs and a topological order \(\Pi_{\alpha,k}\) of the condensation DAG.\\
9 & \qquad For every nontrivial SCC \(C\), factor \(P_CA_C=L_CU_C\); retain scalar substitution for singleton SCCs.\\
10 & \qquad Store \(\mathcal S_{\alpha,k}\) and construct the forward and reverse traversals \(\mathcal B_{\alpha,k}\) and \(\mathcal B_{\alpha,k}^{T}\).\\
11 & \quad \textbf{end for}\\
12 & \quad \textit{Phase II: nonlinear residual and outer stopping test.}\\
13 & \quad \(\mathbf r_k\leftarrow F(\mathbf v_k)-\mathbf d^{obs}\); if the outer criterion holds, \textbf{return} \(\mathbf v_k\).\\
14 & \quad \textit{Phase III: paired matrix-free actions reused throughout LSMR.}\\
15 & \quad For any \(\mathbf p\), and for each \(\alpha\), set \(\delta\mathbf s_{\alpha}\leftarrow M_{\alpha,k}\mathbf p\) and \(\mathbf q_{\alpha}\leftarrow C_{\alpha,k}\delta\mathbf s_{\alpha}\).\\
16 & \quad Apply \(\delta\mathbf T_{\alpha}\leftarrow\mathcal B_{\alpha,k}\mathbf q_{\alpha}\), and return \([J_k\mathbf p]_{\alpha}\leftarrow P_{\alpha,k}\delta\mathbf T_{\alpha}+D_{\alpha,k}\delta\mathbf s_{\alpha}\).\\
17 & \quad For any \(\mathbf y=[\mathbf y_{\alpha}]\), and for each \(\alpha\), apply \(\boldsymbol\lambda_{\alpha}\leftarrow\mathcal B_{\alpha,k}^{T}P_{\alpha,k}^{T}\mathbf y_{\alpha}\).\\
18 & \quad Form \(\mathbf g_{\alpha}\leftarrow M_{\alpha,k}^{T}(C_{\alpha,k}^{T}\boldsymbol\lambda_{\alpha}+D_{\alpha,k}^{T}\mathbf y_{\alpha})\), and return \(J_k^{T}\mathbf y\leftarrow\sum_{\alpha=1}^{N_s}\mathbf g_{\alpha}\).\\
19 & \quad Define \(\mathcal K_k\) and \(\mathcal K_k^T\) from \cref{eq:augmented-actions} using these paired actions.\\
20 & \quad \(\delta\mathbf v_k\leftarrow\operatorname{LSMR}(\mathcal K_k,\mathcal K_k^T,[-W_d\mathbf r_k;\,\mathbf0;\,\mathbf0])\).\\
21 & \quad \textit{Phase IV: nonlinear globalization and model refresh.}\\
22 & \quad \(t_k\leftarrow\operatorname{Armijo}(\mathbf v_k,\delta\mathbf v_k)\), evaluated with nonlinear forward solves.\\
23 & \quad \(\mathbf v_{k+1}\leftarrow\mathbf v_k+t_k\delta\mathbf v_k\); discard the old source setups after the step is accepted.\\
24 & \textbf{end for}; \textbf{return} the last accepted model.\\
\end{tabularx}
\end{algorithm}

Phases I and II are executed once for each outer model. Phase III may be invoked many times by LSMR, but it performs only stored block traversals, local substitutions, and the already frozen source maps. The same discrete objects are used in both directions, which preserves the algebraic transpose pairing. Once the line search accepts a new model, the frozen source states are no longer reused and Phase I rebuilds them at the next outer iteration.

\section{Numerical Experiments}\label{numerical-experiments}

We begin the numerical study with a manufactured point-source problem for which the continuous tangent solution is known, allowing the spatial discretization to be assessed independently of errors from the nonlinear state solve. We then vary the grid deformation and wavefront geometry to examine how the full metric changes the dependency ordering of the frozen transport. The same transport construction is subsequently used in a three-dimensional tomography example, where the cost of repeated SCC/BTF applications is compared directly with converged directional sweeping before both approaches are embedded in the complete matrix-free Gauss-Newton inversion summarized in Algorithm~\ref{alg:overall-gn}.

\subsection{Accuracy and discrete consistency on a deformed grid}\label{accuracy-and-discrete-consistency-on-a-deformed-grid}

We first consider the logical domain $\xi,\eta\in[-L/2,L/2]$ and $\zeta\in[0,H]$, with $L=4$ km and $H=2$ km. The physical grid is generated by
\begin{equation}\label{eq:test-map}
x=\xi,\qquad y=\eta,\qquad z=\zeta+\left(1-\frac{\zeta}{H}\right)h(\xi,\eta),
\end{equation}
where
\begin{equation}\label{eq:test-topo}
h(\xi,\eta)=A\sin\left(\frac{2\pi\xi}{L}\right)\sin\left(\frac{2\pi\eta}{L}\right),\qquad A=200\ \mathrm{m}.
\end{equation}
The upper boundary follows the sinusoidal surface in \cref{eq:test-topo}, whereas the bottom boundary remains flat. The Jacobian determinant is positive on every grid used below. To measure the strength of the nonorthogonality, we define
\begin{equation}\label{eq:muG}
\mu_G=\max_{\mathbf{x}}\max\left\{
\frac{|G_{\xi\eta}|}{\sqrt{G_{\xi\xi}G_{\eta\eta}}},
\frac{|G_{\xi\zeta}|}{\sqrt{G_{\xi\xi}G_{\zeta\zeta}}},
\frac{|G_{\eta\zeta}|}{\sqrt{G_{\eta\eta}G_{\zeta\zeta}}}
\right\}.
\end{equation}
Its value remains between 0.2999 and 0.3083 over the four refinement levels, so refinement does not remove the cross-metric coupling present in the test.

The background velocity is homogeneous, $v_0=3000\ \mathrm{m\,s^{-1}}$, with $s_0=v_0^{-1}$. A point source is placed at the center logical node, corresponding to $\mathbf{x}_s=(0,0,1000\ \mathrm{m})$, and the exact first-arrival field is
\begin{equation}\label{eq:test-T}
T(\mathbf{x})=s_0 r,\qquad r=\|\mathbf{x}-\mathbf{x}_s\|_2.
\end{equation}
We prescribe the smooth slowness perturbation
\begin{equation}\label{eq:test-ds}
\delta s(\mathbf{x})=s_0\left[c_0+c_x\frac{x-x_s}{L}+c_y\frac{y-y_s}{L}+c_z\frac{z-z_s}{H}\right],
\end{equation}
with $(c_0,c_x,c_y,c_z)=(0.02,0.04,-0.03,0.02)$. In a homogeneous background the characteristic from $\mathbf{x}_s$ to $\mathbf{x}$ is a straight segment, and \cref{eq:tangent} reduces along that segment to $d(\delta T)/dr=\delta s$. Since \cref{eq:test-ds} is linear in space, the exact tangent solution is
\begin{equation}\label{eq:test-dT}
\delta T_{\mathrm{ex}}(\mathbf{x})=
\int_0^r \delta s\,d\rho
=\frac{r}{2}\left[\delta s(\mathbf{x}_s)+\delta s(\mathbf{x})\right].
\end{equation}

To isolate the tangent discretization from nonlinear state-solver error, the analytical background field in \cref{eq:test-T} is sampled directly at the grid nodes. The full metric, face-based upwind coefficients, row normalization, and calibrated right-hand side are then formed exactly as in Sections~2--3. The source seed, one outer node layer, and the shrinking neighborhood $r\le 3h$ are excluded from the error norm, with $h$ the maximum logical-grid spacing. For the retained set $\mathcal C_h$, we report
\begin{equation}\label{eq:test-errors}
E_2(h)=\frac{\|\delta T_h-\delta T_{\mathrm{ex}}\|_{2,\mathcal C_h}}{\|\delta T_{\mathrm{ex}}\|_{2,\mathcal C_h}},\qquad
E_\infty(h)=\frac{\|\delta T_h-\delta T_{\mathrm{ex}}\|_{\infty,\mathcal C_h}}{\|\delta T_{\mathrm{ex}}\|_{\infty,\mathcal C_h}},
\end{equation}
and the observed order between successive grids,
\begin{equation}\label{eq:test-order}
p(h_c,h_f)=\frac{\log[E(h_c)/E(h_f)]}{\log(h_c/h_f)}.
\end{equation}

Table~\ref{tab:refinement} presents the refinement results together with three algebraic checks. The normalized solve residual tests the frozen linear solve. The inverse-pair error is
\begin{equation}\label{eq:inverse-pair}
e_{\mathrm{inv}}=
\frac{|\langle y,A^{-1}b\rangle-\langle A^{-T}y,b\rangle|}
{\max\{|\langle y,A^{-1}b\rangle|,|\langle A^{-T}y,b\rangle|,1\}},
\end{equation}
for independent random vectors $b$ and $y$. The final diagnostic, $e_{\mathrm{scale}}$, uses the uniform velocity-scaling identity derived in Section~3.2.

\begin{table}[t]
\caption{Grid refinement and algebraic checks for the full-metric point-source tangent test. Errors $E_2$ and $E_\infty$ are measured on $\mathcal C_h$; observed orders are given in parentheses.}\label{tab:refinement}
\centering\small
\resizebox{\textwidth}{!}{%
\begin{tabular}{ccccccc}
\toprule
Grid & $h$ (m) & $E_2$ ($p_2$) & $E_\infty$ ($p_\infty$) & $r_{\mathrm{solve}}$ & $e_{\mathrm{inv}}$ & $e_{\mathrm{scale}}$\\
\midrule
$21\times21\times11$ & 200 & $1.9062\times10^{-2}$ (--) & $2.1271\times10^{-2}$ (--) & $8.66\times10^{-16}$ & $1.66\times10^{-15}$ & $4.35\times10^{-9}$\\
$41\times41\times21$ & 100 & $9.8331\times10^{-3}$ (0.9550) & $1.1646\times10^{-2}$ (0.8691) & $1.71\times10^{-15}$ & $3.56\times10^{-16}$ & $2.50\times10^{-9}$\\
$81\times81\times41$ & 50 & $5.0761\times10^{-3}$ (0.9539) & $6.2239\times10^{-3}$ (0.9039) & $3.39\times10^{-15}$ & $1.01\times10^{-14}$ & $1.60\times10^{-9}$\\
$161\times161\times81$ & 25 & $2.5937\times10^{-3}$ (0.9687) & $3.3065\times10^{-3}$ (0.9125) & $6.78\times10^{-15}$ & $7.80\times10^{-15}$ & $8.30\times10^{-10}$\\
\bottomrule
\end{tabular}}
\end{table}

The spatial errors decrease monotonically with refinement. The $L^2$ orders remain close to one and the $L^\infty$ order increases toward one, which is consistent with the expected behavior of the face-based first-order upwind tangent discretization on this nonorthogonal grid. We do not infer a general convergence theorem for arbitrary deformations from this manufactured example. The source and boundary exclusions do not create the observed trend: on the finest grid the all-node relative $L^2$ error is essentially the same as that measured on $\mathcal C_h$.

The algebraic errors are much smaller than the spatial discretization error. The solve residual remains near machine precision, as do the inverse-pair and direct $A/A^T$ dot-product tests. The scaling check also decreases with refinement after accounting for the retained single-precision coefficient. These results indicate that the error in $\delta T_h$ is controlled by the spatial discretization rather than by an inconsistent transpose or an inaccurate solution of the frozen linear system. The same pairing is tested again in Section~5.3 after the remaining model, source, receiver, and distributed accumulation operators are included.

\subsection{Dependency structure on full-metric deformed grids}\label{dependency-structure-on-full-metric-deformed-grids}

We next isolate the ordering effect of the full metric. The metric construction, face-based coefficients, row normalization, and dependency orientation are kept unchanged, while the surface deformation and wavefront geometry are varied in a controlled manner. The base domain is $4\ \mathrm{km}\times4\ \mathrm{km}\times2\ \mathrm{km}$ and contains $41\times41\times25$ nodes. Its upper surface is defined by
\begin{equation}\label{eq:graph-map}
z=\zeta+\left(1-\frac{\zeta}{H}\right)A
\sin\left(\frac{2\pi x}{L_x}\right)
\sin\left(\frac{2\pi y}{L_y}\right),
\end{equation}
with $A$ varied from zero to 600 m. The cross-metric strength is again measured by \cref{eq:muG}.

Three surface-source locations are considered. The first family uses a homogeneous background with radial traveltime
\begin{equation}\label{eq:radial-state}
T(\mathbf{x})=\frac{\|\mathbf{x}-\mathbf{x}_s\|}{v_0}.
\end{equation}
A second family changes the wavefront geometry independently of the grid deformation by prescribing
\begin{equation}\label{eq:curved-state}
T(\mathbf{x})=\frac{r}{v_0}
+\epsilon_T\sin\left(\frac{4\pi X}{L_x}\right)
\sin\left(\frac{4\pi Y}{L_y}\right)
\sin\left(\frac{\pi z}{H}\right)-C_s,
\end{equation}
where $X=x+L_x/2$, $Y=y+L_y/2$, and $C_s$ enforces $T(\mathbf{x}_s)=0$. For $\epsilon_T>0$, slowness is defined from $s=\|\nabla_{\mathbf{x}}T\|$, so the continuous isotropic Eikonal equation is satisfied by construction away from the source. Combining three values of $\epsilon_T$, seven deformation amplitudes, and three source locations gives 63 frozen graphs.

For each graph we measure
\begin{equation}\label{eq:graph-fractions}
f_{\mathrm{vio}}=\frac{N_{\mathrm{vio}}}{N_E},\qquad
f_{\mathrm{cyc}}=\frac{N_{\mathrm{cyc}}}{N},\qquad
f_{\mathrm{unr}}=\frac{N_{\mathrm{Kahn\,unresolved}}}{N}.
\end{equation}
Here $N_{\mathrm{vio}}$ counts dependencies that oppose increasing background traveltime, $N_{\mathrm{cyc}}$ counts vertices in nontrivial SCCs, and $N_{\mathrm{Kahn\,unresolved}}$ is the residual population after scalar Kahn elimination. Thus $f_{\mathrm{vio}}$ measures failure of one proposed scalar key, whereas $f_{\mathrm{cyc}}$ measures the part that cannot be made scalar triangular by any permutation.

\subsubsection{Traveltime-order loss and scalar reorderability}\label{traveltime-order-loss-and-scalar-reorderability}

All nine flat-grid cases are traveltime ordered and acyclic. Deformation changes this immediately: 52 of the 54 deformed cases contain order-violating dependencies while remaining DAGs, and only two develop nontrivial SCCs. The distinction predicted in Section~3 is therefore visible in the controlled calculations. Loss of the traveltime key is common, whereas closed algebraic feedback is much more restrictive.

Table~\ref{tab:regimes} gives three representative examples. Case A is the Cartesian radial limit. Case B uses the strongest deformation in the radial family; a substantial fraction of its dependencies violate the traveltime order, yet Kahn elimination still processes the entire graph. Case C has a comparable fraction of order-violating edges but a curved wavefront, and nontrivial SCCs appear. The progression from A to C separates a traveltime-ordered DAG, a reorderable DAG, and a locally irreducible graph without changing the discrete transport construction.

\begin{table}[t]
\caption{Representative dependency regimes on the controlled $41\times41\times25$ grid.}\label{tab:regimes}
\centering\scriptsize
\resizebox{\textwidth}{!}{%
\begin{tabular}{cccccccccc}
\toprule
Case & State & $A$ (m) & $\epsilon_T$ (s) & $\mu_G$ & $f_{\mathrm{vio}}$ (\%) & nontrivial SCCs & $f_{\mathrm{cyc}}$ (\%) & largest SCC & $f_{\mathrm{unr}}$ (\%)\\
\midrule
A & radial & 0 & 0 & $\approx0$ & 0 & 0 & 0 & -- & 0\\
B & radial & 600 & 0 & 0.689 & 6.993 & 0 & 0 & -- & 0\\
C & curved Eikonal & 500 & 0.050 & 0.621 & 6.340 & 6 & 0.195 & 20 & 14.991\\
\bottomrule
\end{tabular}}
\end{table}

The local sign mechanism can also be checked directly. Every order-violating face satisfies the reversal criterion in \cref{eq:reversal} and the corresponding cross-metric dominance relation derived from it. The observed reversals are therefore produced by the full-metric coupling itself rather than by a graph-labeling tolerance. At the same time, case B shows why such reversals should not be called cycles: they can be removed completely by another scalar ordering.

\subsubsection{Local SCCs and downstream ordering obstruction}\label{local-sccs-and-downstream-ordering-obstruction}

Case C illustrates the scale separation described by \cref{prop:kahn}. Only 82 vertices belong to nontrivial SCCs, whereas scalar Kahn elimination leaves 6,300 vertices unresolved. Most of that residual is not cyclic; it lies downstream of the cyclic components and therefore cannot be eliminated until those components are resolved. The unresolved-to-cyclic ratio is 76.8 for the base grid.

The same separation persists under refinement. As shown in Table~\ref{tab:scc-refinement}, the fraction of vertices in nontrivial SCCs decreases across the four grids, whereas the Kahn-unresolved fraction stays close to 15\%. The changing SCC count and block size are not interpreted as an asymptotic geometric law. The relevant observation is narrower: refinement over the tested range does not remove the distinction between a localized irreducible core and its much larger downstream closure.

\begin{table}[t]
\caption{Resolution dependence of the representative localized-SCC case ($A=500$ m, $\epsilon_T=0.050$ s, quarter source).}\label{tab:scc-refinement}
\centering\scriptsize
\resizebox{\textwidth}{!}{%
\begin{tabular}{ccccccccc}
\toprule
Grid & Nodes & nontrivial SCCs & cyclic nodes & $f_{\mathrm{cyc}}$ (\%) & largest SCC & Kahn unresolved & $f_{\mathrm{unr}}$ (\%) & unresolved/cyclic\\
\midrule
$21\times21\times13$ & 5,733 & 2 & 32 & 0.558 & 16 & 897 & 15.646 & 28.0\\
$41\times41\times25$ & 42,025 & 6 & 82 & 0.195 & 20 & 6,300 & 14.991 & 76.8\\
$61\times61\times37$ & 137,677 & 8 & 244 & 0.177 & 42 & 20,643 & 14.994 & 84.6\\
$81\times81\times49$ & 321,489 & 10 & 502 & 0.156 & 64 & 47,978 & 14.924 & 95.6\\
\bottomrule
\end{tabular}}
\end{table}

Two additional checks were made to ensure that the nontrivial components are not created solely by retained precision or negligible coefficients. Small random perturbations of the stored traveltime leave the baseline SCC structure unchanged over a broad range, and progressively removing weak dependencies does not eliminate all nontrivial SCCs until the cutoff is made much larger than roundoff-level weights. These tests support the interpretation of the cyclic blocks as a structural feature of the frozen full-metric transport in this example.

\subsection{Repeated operator applications in a three-dimensional tomography problem}\label{repeated-operator-applications-in-a-three-dimensional-tomography-problem}

We next consider the three-dimensional tomography model in Figure~\ref{fig:true-model}. The boundary-conforming grid contains $121\times121\times56$ nodes, corresponding to 819,896 transport unknowns and 792,000 inversion cells. The target contains four compact velocity perturbations beneath an irregular surface, and the acquisition consists of 100 sources and one million first-arrival traveltimes. Independent zero-mean Gaussian noise is added to the synthetic traveltimes. The data uncertainty combines a $0.3$ ms absolute floor and a $0.3\%$ relative term in quadrature, and each datum is weighted by the inverse of its standard deviation. All timing tests are performed on an AMD EPYC 9654 96-Core Processor using 12 MPI processes with 8 OpenMP threads per process. Sources are distributed across MPI processes, while both the nonlinear wavefront solve and the linearized SCC transport use intra-source OpenMP parallelism. The inner linearized Gauss-Newton systems are solved using LSMR with a Krylov optimality tolerance of $10^{-4}$.

\begin{figure}[t]
\centering
\includegraphics[width=0.95\textwidth]{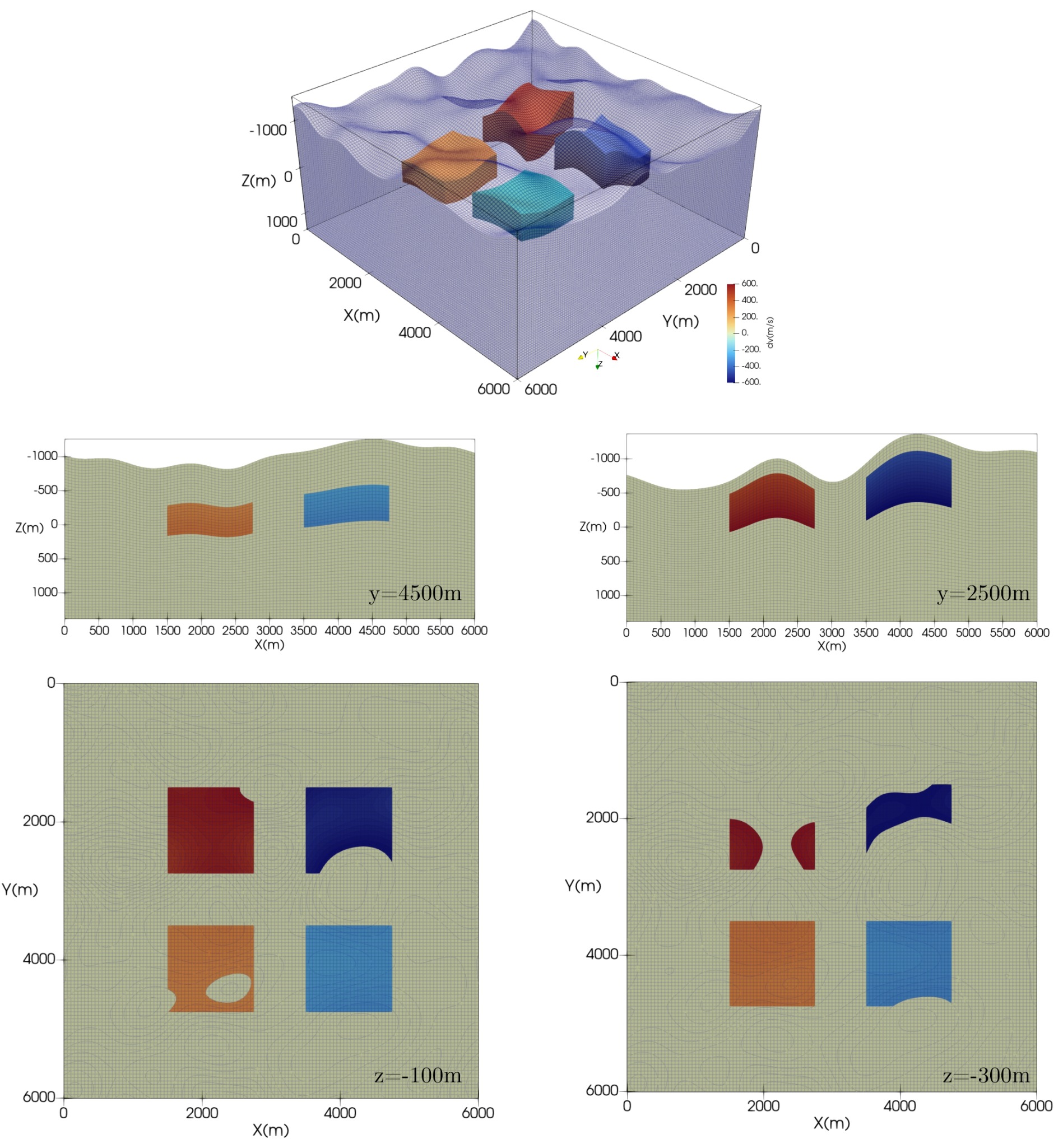}
\caption{Synthetic three-dimensional velocity-perturbation model on the boundary-conforming deformed grid. The upper panel shows the three-dimensional target and surface deformation; the middle and lower rows show representative vertical and horizontal sections, respectively. Colors denote velocity perturbation $\Delta v$ in $\mathrm{m\,s^{-1}}$.}\label{fig:true-model}
\end{figure}

The initial source-dependent graphs retain the same structural separation observed in the controlled examples. Traveltime-order violations occur for every source and nontrivial SCCs occur for most sources, yet only a small fraction of all source-node instances belongs to those SCCs. Scalar Kahn elimination leaves a much larger downstream population unresolved. The structure rows of Table~\ref{tab:production} quantify this contrast.

Before timing the two transport procedures, we compare their actions on identical deterministic directions. The end-to-end dot-product test includes model-to-slowness conversion, source terms, transport, receiver sampling, and distributed accumulation. The SCC/BTF and converged sweeping implementations agree to tight relative accuracy for both $Jp$ and $J^Tq$, so the timing comparison concerns two solution procedures for the same frozen operator.

\begin{table}[t]
\caption{Graph structure, algebraic consistency, and repeated-application timing at the initial model. Source-node fractions are accumulated over the 100 source-specific graphs; wall times use the slowest MPI rank.}\label{tab:production}
\centering\footnotesize
\begin{tabularx}{\textwidth}{@{}p{0.16\textwidth}X>{\raggedleft\arraybackslash}p{0.25\textwidth}@{}}
\toprule
Category & Quantity & Result\\
\midrule
Structure & Traveltime-order-violating dependencies & $9{,}611{,}332/243{,}148{,}592=3.953\%$\\
Structure & Source graphs with nontrivial SCCs & $97/100$\\
Structure & Source-node instances in nontrivial SCCs & $19{,}188/81{,}989{,}600=0.0234\%$\\
Structure & Largest nontrivial SCC & 690 nodes\\
Structure & Kahn-unresolved source-node instances & $42{,}308{,}727/81{,}989{,}600=51.6\%$\\
\addlinespace
Verification & End-to-end $J/J^T$ dot-product relative error & $6.08\times10^{-16}$\\
Verification & Relative difference, SCC/BTF vs. sweeping, $Jp$ & $6.95\times10^{-14}$\\
Verification & Relative difference, SCC/BTF vs. sweeping, $J^Tq$ & $6.94\times10^{-12}$\\
\addlinespace
Performance & One-time graph/SCC setup, maximum rank & 1.02 s\\
Performance & Paired $Jp+J^Tq$ wall time & 84.70 s / 0.313 s\\
Performance & Paired application speedup & $270.3\times$\\
\bottomrule
\end{tabularx}
\end{table}

One paired $Jp+J^Tq$ application requires 84.70 s with directional sweeping and 0.313 s after SCC/BTF setup. The maximum-rank setup cost is 1.02 s, so the setup is recovered within the first paired application in this experiment. The measured ratio is specific to the present grid, model, tolerance, and parallel configuration; its relevance here is the large separation between one-time structural analysis and the repeated cost paid inside a Krylov solve.

\subsection{Tomographic inversion}\label{tomographic-inversion}

The final experiment embeds the reusable transport in the complete matrix-free inversion. The roughness and relative-increment stabilization described in Section~4.5 are used with a trade-off parameter of $0.01$ and $\gamma=0.1$. The data fit is monitored with the weighted RMS residual, using the data standard deviations defined above. The outer iteration is stopped when the relative WRMS reduction between successive Gauss-Newton iterations falls below $1\%$. This criterion is first satisfied after eight iterations, with a final WRMS of $1.0372$. Figure~\ref{fig:recovered} shows the recovered model obtained with the SCC/BTF implementation.

\begin{figure}[t]
\centering
\includegraphics[width=0.95\textwidth]{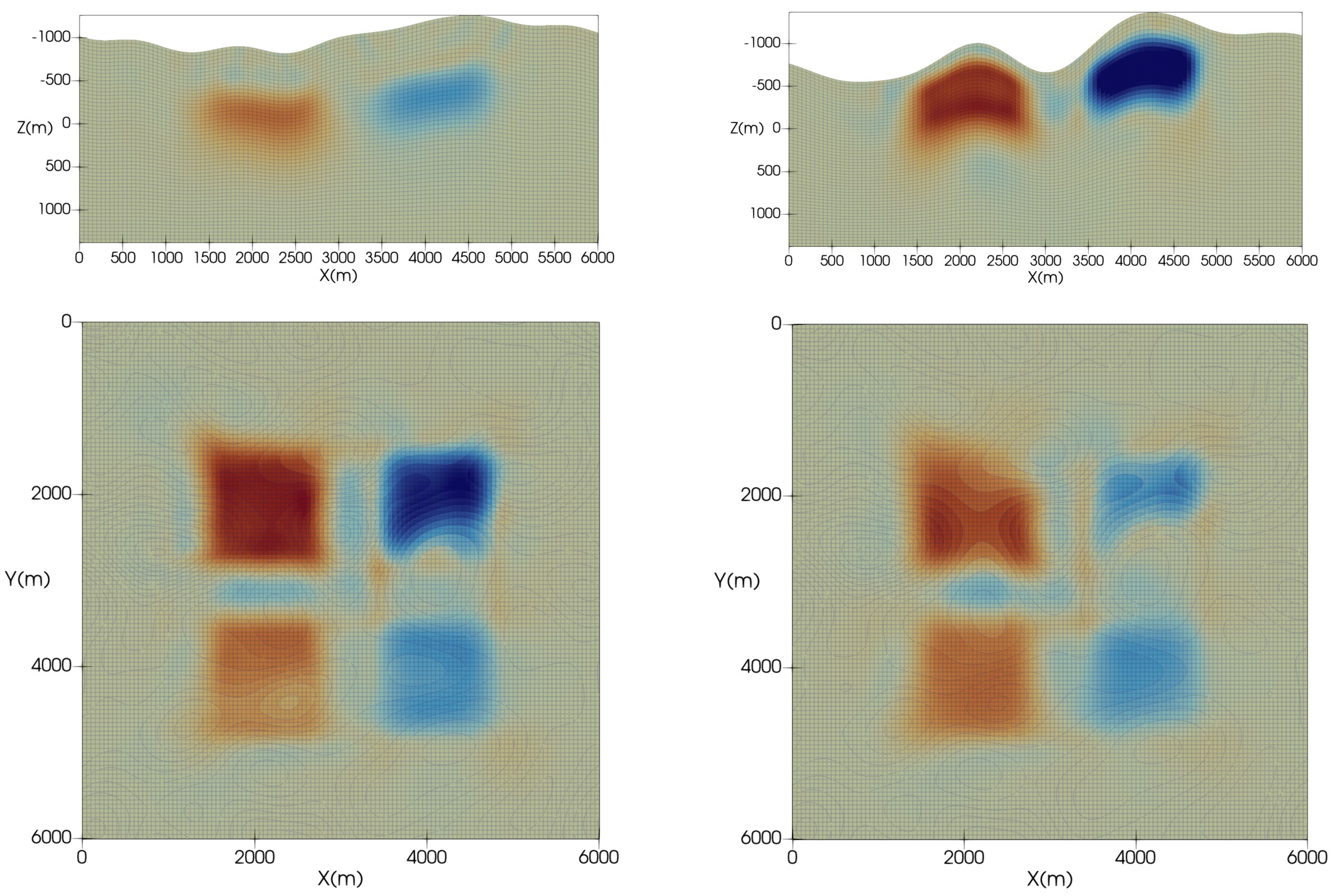}
\caption{Recovered velocity perturbations using the reusable SCC/BTF transport. Representative vertical and horizontal sections are shown with the same velocity-perturbation scale as Figure~\ref{fig:true-model}.}\label{fig:recovered}
\end{figure}

The recovered model reproduces the location and sign of the target perturbations, with the expected smoothing of sharp interfaces from the quadratic increment regularization. The direct operator tests in Table~\ref{tab:production} already establish the numerical equivalence of SCC/BTF and converged sweeping for the frozen primal and transpose actions. A second, visually redundant reconstruction from the sweeping implementation is therefore omitted.

Figure~\ref{fig:convergence-fit} summarizes the inversion behavior and the traveltime fit for the 20th source. The first panel shows the WRMS history together with the number of Krylov iterations used at each outer step; the remaining panels compare the observed and final predicted traveltimes and show their residuals.

\begin{figure}[t]
\centering
\includegraphics[width=\textwidth]{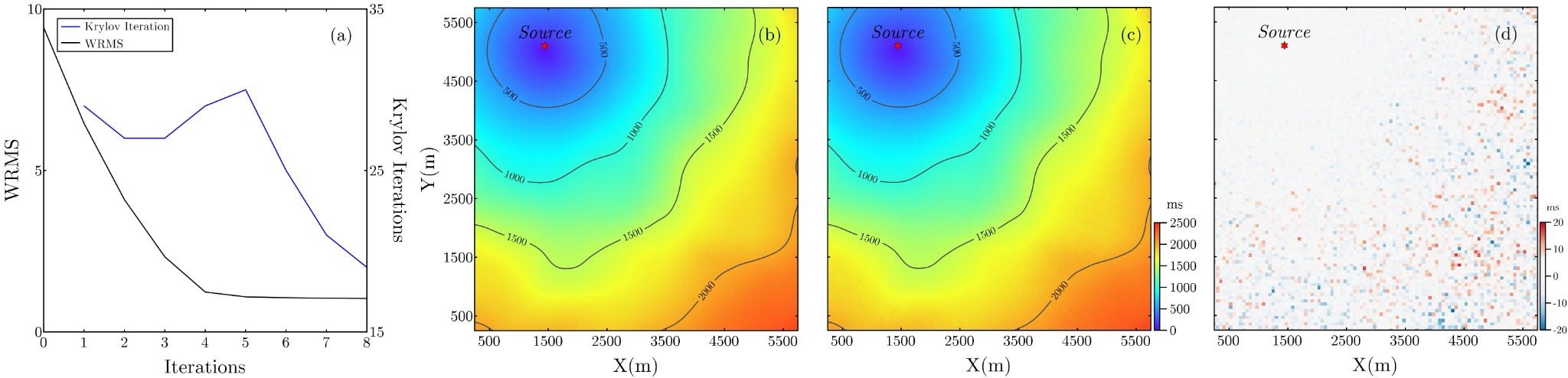}
\caption{Convergence behavior and traveltime fitting results for the three-dimensional tomography experiment. (a) Evolution of the weighted RMS data residual and the number of Krylov iterations over the Gauss-Newton iterations. (b) Observed first-arrival traveltimes for the 20th source. (c) Predicted traveltimes for the same source at the final model. (d) Corresponding traveltime residuals.}\label{fig:convergence-fit}
\end{figure}

For a fixed-work end-to-end timing comparison, both transport implementations were also run for the same 20 Gauss-Newton iterations with otherwise identical settings. Repeated sweeping required 10,744 s, whereas reusable SCC/BTF transport required 607 s, an overall wall-clock speedup of 17.7. This factor is smaller than the paired-operator ratio in Table~\ref{tab:production} because nonlinear forward solves, line-search evaluations, regularization, Krylov vector operations and communication, and model output are common to both calculations.

\section{Conclusions and outlook}\label{conclusions}

Matrix-free Gauss-Newton traveltime tomography avoids explicit Jacobian storage but repeatedly applies a frozen sensitivity transport and its transpose. On boundary-conforming deformed grids, full-metric coupling can invalidate background traveltime as a scalar ordering even though continuous first-arrival causality is preserved. Repeating directional sweeps therefore spends substantial work rediscovering an ordering problem that does not change during the inner Krylov solve.

The graph formulation developed here separates ordering loss from true irreducibility. SCC condensation recovers an exact block topological order, so scalar substitution is retained on singleton components and only nontrivial SCCs require local block solves. The resulting graph, ordering, and local factors are reusable for both primal and transpose actions. Controlled experiments show that traveltime-order violations can occur without cycles and that a small cyclic core can obstruct a much larger downstream scalar elimination region. In the three-dimensional example, the irreducible components remain localized and the reusable block traversal reproduces the converged sweeping actions while greatly reducing their repeated cost.

The present dense local treatment is intended for problems in which nontrivial SCCs remain modest. Larger or more strongly coupled components would motivate sparse or iterative block solvers and more explicit parallel scheduling on the condensation DAG. The broader principle is independent of those local choices: for a matrix-free inverse problem that repeatedly applies a frozen upwind transport operator, its dependency structure can be analyzed once and reused rather than reconstructed for every right-hand side.

\FloatBarrier
\section*{Acknowledgments}
The authors gratefully acknowledge Tongji University for providing the high-performance computing resources used in this work.

\end{document}